\documentclass[11pt]{article}
\usepackage{amsmath,amssymb,amsfonts}
\usepackage{amsthm}
\usepackage{cite}
\newtheorem{theorem}{Theorem}

\newtheorem{lemma}{Lemma}

\title{\textbf{Moderate Deviation Principle for a Class of SPDEs} }
\date{}
\author{Parisa Fatheddin$^{a}$ and Jie Xiong$^{b,}$\thanks{Research supported partially by
FDCT 076/2012/A3.}\\
\footnotesize ${ }^{a}$Department of Mathematics, University of Alabama, Huntsville, Al 35899, USA.\\
\footnotesize ${ }^{b}$Department of Mathematics, FST, University of Macau, PO Box 3001, Macau, China
}

\begin{document}
\newtheorem{example}[theorem]{Example}
\newtheorem{cor}[theorem]{Corollary}
\newtheorem{notation}[theorem]{Notation}
\newtheorem{notations}[theorem]{Notations}
\newtheorem{claim}[theorem]{Claim}
\newtheorem{mtheorem}[theorem]{Meta-Theorem}
\newtheorem{prop}[theorem]{Proposition}
\newtheorem{rem}[theorem]{Remark}
\newtheorem{conj}[theorem]{Conjecture}
\newtheorem{rems}[theorem]{Remarks}
\maketitle
\begin{abstract}
We establish the moderate deviation principle for the solutions of a class of stochastic partial differential equations with non-Lipschitz continuous coefficients. As an application, we derive the moderate deviation principle for two important population models: super-Brownian motion and Fleming-Viot process.
\end{abstract}
\noindent {\sc Key words:} Moderate deviation principle, stochastic
partial differential
 equation, Fleming-Viot process, super-Brownian motion.

\noindent {\sc MSC 2010 subject classifications:} Primary 60F10;
Secondary: 60H15, 60J68.

\section{Introduction}
Many problems in the fields of applications can be modeled by measure-valued processes. Among them are two of the most commonly studied population models called super-Brownian motion (SBM) and Fleming-Viot process (FVP). These population models have been the focus of numerous recent publications. One of the interesting problems on the models is to set the branching rate for SBM and resampling rate for FVP to tend to zero and to study the rate at which the population's measure converges to a deterministic limit. This rate of convergence is best given by the large deviation principle (LDP). In \cite{me}, we achieved the LDP for SBM and FVP as the above mentioned rates go to zero and obtained an explicit form of the rate of convergence for each model. However, the topology introduced there is not the natural one. Namely, we used the double quotient space due to the non-uniqueness of the controled PDE in the definition of the rate function.  Here we achieve the moderate deviation principle (MDP), which provides the convergence rate of the models as the branching/resampling rate tend to zero at a speed slower than that considered for the LDP. The topology we shall use is the standard one, and there is no need to introduce the quotient space.

MDP for SBM has also been established by Schied in \cite{Schied}. He used space $\mathcal{C}\left([0,1]; M\left(\mathbb{R}^{d}\right)\right)$ equipped with compact open topology, where $M(\mathbb{R}^{d})$ is the space of finite signed measures on $\mathbb{R}^{d}$ with the coarsest topology in which $\mu \mapsto \left<\mu,f\right>$ are continuous for every bounded Lipschitz function on $\mathbb{R}^{d}$. The main tool he applied is the G$\ddot{a}$rtner-Ellis theorem (cf. Theorem 4.6.1 of \cite{Dem}). Here we have used a similar space and have obtained the same result; however, with a different approach. Other authors including those of \cite{HongMDP, Hong, Yang, Zhang} have investigated the MDP for processes related to SBM. These processes include SBM with super-Brownian immigration (SBMSBI) and SBM with immigration governed by Lebesgue measure. Here we have also derived the MDP for FVP, which to our knowledge, has not yet been shown in the literature.

In this article we study the SBM and FVP based on their characterization by solutions to certain SPDEs. We formulate a general class of SPDEs by observing the similarities between the two SPDEs and in Section 3 derive the MDP for this class by applying Theorem 6 of \cite{Bud}. In Section 4, we then establish the MDP for the two population models with the help of the contraction principle (cf. \cite{Dup} Theorem 4.2.1). We note that since the formulation of SBM and FVP by SPDEs offered by \cite{Xio} was given only for dimension one then our result on the MDP is limited to this dimension. For higher dimensions further investigation is required.

\section{Notations and Main Results}
Suppose $(\Omega, \mathcal{F}, P)$ is a probability space and $\{\mathcal{F}_{t}\}$ is a family of non-decreasing right continuous sub-$\sigma$-fields of $\mathcal{F}$ such that $\mathcal{F}_{0}$ contains all $P$-null subsets of $\Omega$. We denote $\mathcal{C}_{b}(\mathbb{R})$ to be the space of continuous bounded functions on $\mathbb{R}$ and $\mathcal{C}_{c}(\mathbb{R})$ be the set of continuous functions in $\mathbb{R}$ with compact support. In addition, for $0<\beta \in \mathbb{R}$, we let $\mathcal{M}_{\beta}(\mathbb{R})$ denote the set of $\sigma$-finite measures $\mu$ on $\mathbb{R}$ such that
\begin{equation}\label{Mbeta}
\int e^{-\beta |x|} d\mu(x) <\infty.
\end{equation}
 We endow this space with the topology defined by a modification of the usual weak topology: $\mu^{n}\rightarrow \mu$ in $\mathcal{M}_{\beta}(\mathbb{R})$ iff for every $f\in \mathcal{C}_{b}(\mathbb{R})$,
\begin{equation*}
\int_{\mathbb{R}}f(x)e^{-\beta |x|} \mu^{n}(dx) \rightarrow \int_{\mathbb{R}}f(x)e^{-\beta |x|}\mu(dx).
\end{equation*}
This topology is given by the following modified Wasserstein distance,
\begin{eqnarray*}
 &&\rho_{\beta}(\mu,\nu)\\
 &:=& \inf\left\{\left|\int_{\mathbb{R}} f(x)e^{-\beta|x|}\left(\mu(dx)-\nu(dx)\right)\right|:
 f\in \mathcal{C}_{b}^{1}(\mathbb{R}), \|f\|_{\infty}\vee \|f'\|_{\infty}\leq 1\right\}.
 \end{eqnarray*}
 We denote the probability measures on $\mathbb{R}$ with the above topology by $\mathcal{P}_{\beta}(\mathbb{R})$. Let $(S, \mathcal{S})$ be the measurable space defined as
\begin{equation}
(S, \mathcal{S}):= (\mathcal{C}([0,1];\mathbb{R}^{\infty}), \mathbb{B}(\mathcal{C}([0,1];\mathbb{R}^{\infty}))),
\end{equation}
where $\mathbb{R}^{\infty}$ is the Polish space with the metric given as
\begin{equation*}
d(\{x_{i}\},\{y_{i}\}):= \sum_{i=1}^{\infty} 2^{-i}(|x_{i}-y_{i}|\wedge 1).
\end{equation*}

Throughout this paper, we assume $\beta_{0}\in(0,\beta)$ and $K$ is a constant which may take different values in different lines. Also notation $\Delta$ stands for the second derivative in the spatial variable $x$. This notation will be used when both spatial and time variables are involved, or when the dual operator will be needed. Otherwise, we will use the simpler notation $f''$. Same convention is used for $\nabla$, the first derivative in spatial variable. For $\alpha \in (0,1)$, we consider the space $\mathbb{B}_{\alpha, \beta}$ composed of all functions $f:\mathbb{R}\rightarrow \mathbb{R}$ such that for every $m\in \mathbb{N}$, there exist constants $K>0$ with the following conditions:
\begin{eqnarray}
\left|f(y_{1})-f(y_{2})\right|&\leq& Ke^{\beta m}|y_{1}-y_{2}|^{\alpha} , \hspace{.4cm} \forall |y_{1}|,|y_{2}| \leq m\\ \label{cond1}
|f(y)| &\leq& Ke^{\beta |y|}, \hspace{.4cm} \forall y\in \mathbb{R}\label{cond2}
\end{eqnarray}
and with the metric,
\begin{equation*}
d_{\alpha,\beta}(u,v)=\sum^\infty_{m=1}2^{-m}(\|u-v\|_{m,\alpha,\beta}\wedge1),\qquad u,\;v\in\mathbb{B}_{\alpha,\beta}
\end{equation*}
 where
 \begin{equation*}
\|u\|_{m,\alpha,\beta}=\sup_{x\in\mathbb{R}}e^{-\beta|x|}|u(x)|+\sup_{y_{1}\neq y_{2}\\ \left|y_{1}\right|,\left|y_{2}\right|\leq m}\frac{|u(y_{1})-u(y_{2})|}{|y_{1}-y_{2}|^\alpha}e^{-\beta m}.
\end{equation*}
Note that the collection of continuous functions on $\mathbb{R}$ satisfying (\ref{cond2}), referred to as $\mathbb{B}_{\beta}$, is a Banach space with norm,
\begin{equation*}
\|f\|_{\beta}= \sup_{x\in \mathbb{R}} e^{-\beta |x|}|f(x)|.
\end{equation*}

For the convenience of the reader, we now offer a quick introduction to the two population models considered. In SBM model, each individual has an exponentially distributed lifetime and the population evolves as a ``cloud.'' It is studied by taking a scaled limit of a branching process with an associated branching rate. SBM with branching rate $\epsilon$, denoted by $\mu_{t}^{\epsilon}$, is a measure-valued Markov process that can be characterized by one of the following.

$i)$ $(\mu_{t}^{\epsilon})$ having Laplace transform,
\begin{equation*}
\mathbb{E}_{\mu_{0}^{\epsilon}} \exp(-\left<\mu_{t}^{\epsilon},f\right>)= \exp(-\left<\mu_{0}^{\epsilon}, v(t,\cdot)\right>),
\end{equation*}
where $v(\cdot, \cdot)$ is the unique mild solution of the evolution equation:
\begin{equation*}
\left\{\begin{array}{l} \dot{v}(t,x)= \frac{1}{2}\Delta v(t,x) - v^{2}(t,x),  \\
v(0,x)=f(x),
\end{array}\right.
\end{equation*}
for $f\in \mathcal{C}_{p}^{+}(\mathbb{R}^{d})$ where for $K>0$,
\begin{equation*}
\mathcal{C}_{p}(\mathbb{R}^{d}):= \left\{f\in \mathcal{C}(\mathbb{R}^{d}): |f(x)|< K\phi_{p}(x)\mbox{ for } p>d, \phi_{p}(x):= (1+|x|^{2})^{-\frac{p}{2}}\right\}.
\end{equation*}.

$ii)$ $(\mu_{t}^{\epsilon})$ as the unique solution to a martingale problem given as: for all $f\in \mathcal{C}_{b}^{2}(\mathbb{R})$
\begin{equation*}
M_{t}(f):= \left<\mu_{t}^{\epsilon},f\right>-\left<\mu_{0}^{\epsilon},f\right>-\int_{0}^{t}\left<\mu_{s}^{\epsilon},\frac{1}{2}\Delta f\right> ds,
\end{equation*}
is a square-integrable martingale with quadratic variation,
\begin{equation*}
\left<M(f)\right>_{t}= \epsilon \int_{0}^{t} \left<\mu_{s}^{\epsilon},f^{2}\right>ds.
\end{equation*}

$iii)$ In \cite{Xio} SBM was studied by its ``distribution'' function-valued process $u_{t}^{\epsilon}$ defined as
\begin{equation}\label{SBM d}
u_{t}^{\epsilon}(y)=\int_{0}^{y}\mu_{t}^{\epsilon}(dx), \hspace{1cm} \forall y\in \mathbb{R}
\end{equation}
and using (\ref{SBM d}), SBM was characterized by the following stochastic partial differential equation (SPDE),
\begin{equation}\label{SBM}
u_{t}^{\epsilon}(y)=F(y)+ \int_{0}^{t}\int_{0}^{u_{s}^{\epsilon}(y)} W(dsda) + \int_{0}^{t} \frac{1}{2} \Delta u_{s}^{\epsilon}(y)ds,
\end{equation}
where $F(y)=\int_{0}^{y}\mu_{0}(dx)$ is the ``distribution'' function of $\mu_{0}$, $W$ is an $\mathcal{F}_{t}$-adapted space-time white noise random measure on $\mathbb{R}^{+}\times \mathbb{R}$ with intensity measure $dsda$.

On the other hand, FVP is a population model with its evolution based on the genetic types of the individuals. It is a probability measure-valued diffusion process studied as a scaled limit of a step-wise mutation model, in which the population size is assumed to stay constant throughout time and individuals move in $\mathbb{Z}^{d}$ according to a continuous time simple random walk. As in the case for SBM, this population model denoted as $(\mu_{t}^{\epsilon})$, is a Markov process and can be characterized by one of the following.

$i)$ $(\mu_{t}^{\epsilon})$ a family of probability measure-valued Markov process generated by $\mathcal{L}^{\epsilon}$ defined as
\begin{eqnarray*}
\mathcal{L}^{\epsilon}F(\mu_{t}^{\epsilon})&=& f'(\left<\mu_{t}^{\epsilon},\phi\right>)\left<\mu_{t}^{\epsilon},A\phi\right> \\
&&\hspace{.4cm}+ \frac{\epsilon}{2}\int\int f''(\left<\mu_{t}^{\epsilon},\phi\right>)\phi(x)\phi(y)Q(\mu_{t};dx,dy),
\end{eqnarray*}
for $\epsilon>0$ where
\begin{equation*}
Q(\mu_{t}^{\epsilon};dx,dy):= \mu_{t}^{\epsilon}(dx)\delta_{x}(dy)-\mu_{t}^{\epsilon}(dx)\mu_{t}^{\epsilon}(dy),
\end{equation*}
with $\delta_{x}$ denoting the Dirac measure at $x$ and $A$ being the generator of a Feller process. The operator $\mathcal{L}^{\epsilon}$ is given on the set,
\begin{equation*}
\mathcal{D}= \{F(\mu_{t}^{\epsilon})=f(\left<\mu_{t}^{\epsilon},\phi\right>):f\in \mathcal{C}_{b}^{2}(\mathbb{R}), \phi \in \mathcal{C}(\mathbb{R})\}.
\end{equation*}
(see \cite{DF1} and \cite{FenX} for this formulation). \\

$ii)$ $(\mu_{t}^{\epsilon})$ as a unique solution to the following martingale problem: for $f\in \mathcal{C}_{c}^{2}(\mathbb{R})$,
\begin{equation*}
M_{t}(f)= \left<\mu_{t}^{\epsilon},f\right>-\left<\mu_{0}^{\epsilon},f\right>-\int_{0}^{t} \left<\mu_{s}^{\epsilon},\frac{1}{2}\Delta f\right>ds,
\end{equation*}
is a continuous square-integrable martingale with quadratic variation,
\begin{equation*}
\left<M_{t}(f)\right>= \epsilon \int_{0}^{t}\left(\left<\mu_{s}^{\epsilon},f^{2}\right>-\left<\mu_{s}^{\epsilon},f\right>^{2}\right)ds.
\end{equation*}

$iii)$ An alternative formulation of FVP was also made in \cite{Xio}. There by using
\begin{equation}
u_{t}^{\epsilon}(y)= \mu_{t}^{\epsilon}((-\infty,y]),
\end{equation}
FVP was proved to be given by the solution to the following SPDE,
\begin{equation}\label{FVP}
u_{t}^{\epsilon}(y)= F(y) + \int_{0}^{t}\int_{0}^{1} \left(1_{a\leq u_{s}^{\epsilon}(y)}-u_{s}^{\epsilon}(y)\right)W(dsda) + \int_{0}^{t} \frac{1}{2}\Delta u_{s}^{\epsilon}(y)ds.
\end{equation}

Based on the context, $\epsilon >0$ represents the branching rate for SBM and resampling rate for FVP. Note that the main difference between (\ref{SBM}) and (\ref{FVP}) is in the second term; therefore, in \cite{Xio} a general SPDE with small noise term of the form
\begin{equation}\label{SPDE}
u_{t}^{\epsilon}(y)=F(y) + \sqrt{\epsilon} \int_{0}^{t}\int_{U}G(a,y,u_{s}^{\epsilon}(y)) W(dsda) + \int_{0}^{t} \frac{1}{2}\Delta u_{s}^{\epsilon}(y)ds,
\end{equation}
was considered with conditions,
\begin{eqnarray}
\int_{U}\left|G(a,y,u_{1})-G(a,y,u_{2})\right|^{2} \lambda(da) &\leq& K|u_{1}-u_{2}|, \label{con1}\\
\int_{U}|G(a,y,u)|^{2}\lambda(da) &\leq& K(1+|u|^{2}),\label{con2}
\end{eqnarray}
where $u_{1},u_{2},u,y \in \mathbb{R}$, $F$ is a function on $\mathbb{R}$ and $G:U\times \mathbb{R}^{2} \rightarrow \mathbb{R}$. Here we prove the MDP for $\{u_{t}^{\epsilon}\}$ by considering the LDP for $\{v_{t}^{\epsilon}\}$ given by,
\begin{equation}\label{centered}
v_{t}^{\epsilon}(y):= \frac{a(\epsilon)}{\sqrt{\epsilon}} (u_{t}^{\epsilon}(y)-u_{t}^{0}(y)).
\end{equation}
Hence we have,
 \begin{equation}\label{MDP SPDE}
v_{t}^{\epsilon}(y)= a(\epsilon) \int_{0}^{t}\int_{U} G_{s}^{\epsilon}(a,y,v_{s}^{\epsilon}(y))W(dsda) + \frac{1}{2}\int_{0}^{t} \Delta v_{s}^{\epsilon}(y)ds,
\end{equation}
where $G_{s}^{\epsilon}(a,y,v):= G(a,y,\frac{\sqrt{\epsilon}}{a(\epsilon)}v+u_{s}^{0}(y))$ and $a(\epsilon)$ satisfies $0\leq a(\epsilon)\rightarrow 0$ and $\frac{a(\epsilon)}{\sqrt{\epsilon}} \rightarrow \infty$ as $\epsilon \rightarrow 0$. To form the controlled PDE of (\ref{MDP SPDE}) we replace the noise by $h\in L^{2}([0,1]\times U, ds\lambda(da))$ and obtain
\begin{equation}\label{controlled}
v_{t}(y)=\int_{0}^{t}\int_{U}G(a,y,u_{s}^{0}(y))h(s,a)\lambda(da)ds + \frac{1}{2}\int_{0}^{t}\Delta v_{s}(y)ds.
\end{equation}
Note that for every $h\in L^{2}([0,1]\times U, ds\lambda(da))$, SPDE (\ref{controlled}) has a unique solution, which we denote as $\gamma(h)$ for a map $\gamma:L^{2}([0,1]\times U, ds\lambda(da))\rightarrow \mathcal{C}([0,1];\mathbb{B}_{\beta})$. We are now ready to state the first result of this paper.

\begin{theorem}\label{them1}
If $F\in \mathbb{B}_{\alpha,\beta_{0}}$ for $\alpha \in \left(0,\frac{1}{2}\right)$ then family $\{v^{\epsilon}_{.}\}$ given by (\ref{MDP SPDE}) satisfies the LDP in $\mathcal{C}([0,1];\mathbb{B}_{\beta})$ with rate function,
\begin{equation}\label{rate}
I(v)= \frac{1}{2}\inf\left\{\int_{0}^{1}\int_{U}\left|h_{s}(a)\right|^{2}\lambda(da)ds: v = \gamma(h)\right\},
\end{equation}
which implies that family $\{u_{t}^{\epsilon}\}$ obeys the MDP.
\end{theorem}
Similar to \cite{FGK} we consider the Cameron-Martin space which is defined as follows. Let $\mathcal{M}^S_\beta(\mathbb{R})$ be the space of signed measures $\mu=\mu_+-\mu_-$ with $\mu_\pm\in\mathcal{M}_\beta(\mathbb{R})$. Let $\mathcal{D}$ be the Schwartz space of test functions with compact support in $\mathbb{R}$
and continuous derivatives of all orders. Denote the dual space of real distributions on
$\mathbb{R}$ by $\mathcal{D}^{*}$ then the Cameron-Martin space, $H$, is composed of $\omega \in \mathcal{C}([0,1];\mathcal{M}^S_\beta(\mathbb{R}))$ satisfying the conditions below.
\begin{enumerate}
\item  $\omega_{0}=0$,
\item the $\mathcal{D}^{*}$-valued map $t\mapsto \omega_{t}$ defined on [0,1] is absolutely continuous with
 respect to time. Let $\dot{\omega}$ and $\Delta^{*}\omega$ be its generalized derivative and Laplacian
  respectively,
\item for every $t\in [0,1]$, $\dot{\omega}_t - \frac{1}{2}\Delta^{*} \omega_{t} \in \mathcal{D}^{*}$ is
 absolutely continuous with respect to $\mu^0_{t}$
 with $\frac{d\left(\dot{\omega}_{t}-\frac{1}{2}\Delta^{*}\omega_{t}\right)}{d\mu_t^{0}}$
  being the (generalized) Radon Nikodym derivative,
\item $\frac{d\left(\dot{\omega}_{t}-\frac{1}{2}\Delta^{*}\omega_{t}\right)}{d\mu_t^{0}}$ is in $L^{2}([0,1] \times \mathbb{R},
ds\mu(dy))$.
\end{enumerate}

Let $\tilde{H}$ be the space for which
conditions for $H$ hold with $\mathcal{M}^S_\beta(\mathbb{R})$
replaced by the space of  measures $\mathcal{P}^S_\beta(\mathbb{R})$, and with the additional assumption,
\begin{equation*}
\left<\mu_{t}^{0}, \frac{ d \left(\dot{\omega}_{t}-\frac{1}{2}\Delta^{*}\omega_{t}\right)}{d\mu_t^{0}}\right> = 0,
\end{equation*}
where $\mathcal{P}^S_\beta(\mathbb{R})$ is the set of signed measures $\mu$ with $\mu_\pm\in\mathcal{P}_\beta(\mathbb{R})$ and $\mu_\pm(\mathbb{R})=1$.
Denoting $\omega_{t}^{\epsilon}(dy):= \frac{a(\epsilon)}{\sqrt{\epsilon}} \left(\mu_{t}^{\epsilon}(dy)- \mu_{t}^{0}(dy)\right)$ we have the following two theorems.

\begin{theorem}\label{them2}
If $\omega_{0} \in \mathcal{M}_{\beta}(\mathbb{R})$  such that $F\in \mathbb{B}_{\alpha, \beta_{0}}$ then super-Brownian motion, $\{\mu_{t}^{\epsilon}\}$, obeys the MDP in $\mathcal{C}([0,1];\mathcal{M}^S_{\beta}(\mathbb{R}))$ with rate function,
\begin{equation}\label{rate4sbm}
  I(\omega)=  \left\{\begin{array} {ll}  \frac{1}{2} \displaystyle \int_{0}^{1}
  \int_{\mathbb{R}}\left|\frac{d\left(\dot{\omega}_t -\frac{1}{2}\Delta^{*}\omega_{t}\right)}{d\mu_{t}^{0}}(y)\right|^2 \mu_{t}^{0}(dy) dt
   & \mbox{\emph{if }} \mu \in H \\
   \infty
  & \mbox{\emph{otherwise}.}     \end{array}   \right.
 \end{equation}
\end{theorem}

\begin{theorem}\label{them3}
Suppose $\omega_{0} \in \mathcal{P}_{\beta}(\mathbb{R})$ such that $F
\in \mathbb{B}_{\alpha,\beta_{0}}$. Then, Fleming-Viot process, $\{\mu^{\epsilon}\}$,
satisfies the MDP on $\mathcal{C}([0,1];
\mathcal{P}^S_{\beta}(\mathbb{R}))$ with rate function,
\begin{equation}\label{rate4fvp}
  I(\omega)=  \left\{\begin{array} {ll}  \frac{1}{2} \displaystyle \int_{0}^{1}
  \int_{\mathbb{R}}\left|\frac{d \left(\dot{\omega_{t}}-\frac{1}{2}\Delta^{*}\omega_{t}\right)}{d\mu_{t}^{0}}
(y)  \right|^2 \mu_{t}^{0}(dy) dt
   & \mbox{\emph{if }} \mu \in \tilde{H} \\
   \infty
  & \mbox{\emph{otherwise.}}     \end{array}   \right.
 \end{equation}
 \end{theorem}
 Proofs of Theorems 1-3 are given in Sections 3-5.

\section{Moderate Deviations for the General SPDE}
Our goal in this section is to establish the MDP for SPDE (\ref{SPDE}), referred to as the general SPDE. Note that by our assumption $F\in \mathbb{B}_{\alpha,\beta_{0}}$, we have,
\begin{equation*}
|u_{s}^{0}(y)| \leq \int_{\mathbb{R}} p_{s}(x-y)|F(x)|dx\leq Ke^{\beta_{0}|y|},
\end{equation*}
where $p_{t}(x)= \frac{1}{\sqrt{2\pi t}}\exp(-\frac{x^{2}}{2t})$ is the heat kernel. Therefore, $G_{s}^{\epsilon}$ satisfies conditions,
\begin{eqnarray}
\int_{U}\left|G_{s}^{\epsilon}(a,y,v_{1})-G_{s}^{\epsilon}(a,y,v_{2})\right|^{2} \lambda(da) &\leq& K|v_{1}-v_{2}|,\label{MDP con1}\\
\int_{U}|G_{s}^{\epsilon}(a,y,v)|^{2}\lambda(da) &\leq& K(1+v^{2}+ e^{2\beta_{0} |y|}),\label{MDP con2}
\end{eqnarray}
for $y\in \mathbb{R}$ and $v,v_{1},v_{2}\in \mathbb{R}$ given by (\ref{centered}).

Since the proof of the uniqueness of strong solutions to SPDE (\ref{SPDE}) established in \cite{Xio} only uses condition (\ref{con1}) then the same argument can be applied to SPDE (\ref{MDP SPDE}) to achieve the uniqueness of strong solutions. SPDE (\ref{MDP SPDE}) can therefore be presented by its mild form,
\begin{equation}\label{MDP mild}
v_{t}^{\epsilon}(y)= a(\epsilon) \int_{\mathbb{R}}\int_{0}^{t}\int_{U}G_{s}^{\epsilon}(a,x,v_{s}^{\epsilon}(x))p_{t-s}(y-x)W(dsda)dx.
\end{equation}
We show that this mild solution takes values in $\mathcal{C}([0,1];\mathbb{B}_{\beta})$. To accomplish this we need the subsequent lemma.

\begin{lemma}\label{lem1}
For every $n\geq 2$,
\begin{equation}\label{Mlemma}
\tilde{M} := \sup_{0<\epsilon <1}\mathbb{E} \sup_{0\leq s\leq 1}\left(\int_{\mathbb{R}} |v_{s}^{\epsilon}(x)|^{2}e^{-2\beta_{1}|x|}dx\right)^{n} < \infty.
\end{equation}
\end{lemma}

\begin{proof}
We adapt the argument in the proof of Lemma 2.3 of \cite{Xio} to present setup. By Mitoma \cite{Mit} if
\begin{equation*}
  \rho(x)=  \left\{\begin{array} {ll}  C\exp\left(\frac{-1}{1-|x|^{2}}\right)
   &  |x|<1 \\
   0
  & |x|\geq 1,    \end{array}   \right.
 \end{equation*}
where $C$ is determined by $\int_{\mathbb{R}}\rho(x)dx=1$, then $g(x)=\int_{\mathbb{R}} e^{-|y|}\rho(x-y)dy$ satisfies
\begin{equation}\label{g}
Ke^{-|x|}\leq g^{(n)}(x)\leq Ke^{-|x|},
\end{equation}
where $g^{(n)}(x)$ is the $n^{\mbox{th}}$ derivative of $g(x)$. Note that (\ref{g}) stays true with $e^{-|x|}$ replaced by $e^{-2\beta_{1}|x|}$. We then consider $\int J(x)d\mu(x)<\infty$ where $J(x)= \int e^{-2\beta_{1}|y|}\rho(x-y)dy<\infty$ for the definition of $\mathcal{M}_{\beta}(\mathbb{R})$ given by (\ref{Mbeta}).

We denote the Hilbert space $L^{2}\left(\mathbb{R}, J(x)dx\right)$ by $\mathcal{X}_{0}$. Applying It$\hat{o}$'s formula to (\ref{MDP SPDE}) we have for every $f\in \mathcal{C}_{c}^{\infty}(\mathbb{R})\cap \mathcal{X}_{0}$,
\begin{eqnarray}\label{above}
<v_{t}^{\epsilon},f>_{\mathcal{X}_{0}}&=& a(\epsilon)\int_{\mathbb{R}}\int_{0}^{t}\int_{U} G_{s}^{\epsilon}(a,y,v_{s}^{\epsilon}(y))f(y)J(y)W(dsda)dy \nonumber\\
&& \hspace{.3cm} + \int_{0}^{t}<\frac{1}{2} \Delta v_{s}^{\epsilon},f>_{\mathcal{X}_{0}}ds,
\end{eqnarray}
It$\hat{o}$'s formula applied again this time to (\ref{above}) gives,
\begin{eqnarray}
&&<v_{t}^{\epsilon},f>_{\mathcal{X}_{0}}^{2}\\
&&= 2a(\epsilon)\int_{0}^{t} <v_{s}^{\epsilon},f>_{\mathcal{X}_{0}}\int_{U}\int_{\mathbb{R}} G_{s}^{\epsilon}(a,y,v_{s}^{\epsilon}(y))f(y)J(y)dyW(dsda)\nonumber\\
&& \hspace{.3cm} +\int_{0}^{t} <v_{s}^{\epsilon},f>_{\mathcal{X}_{0}}< \Delta v_{s}^{\epsilon},f>_{\mathcal{X}_{0}}ds \nonumber\\
&&\hspace{.3cm} + a(\epsilon)^2\int_{0}^{t}\int_{U}\left(\int_{\mathbb{R}} G_{s}^{\epsilon}(a,y,v_{s}^{\epsilon}(y))f(y)J(y)dy\right)^{2}\lambda(da)ds. \nonumber
\end{eqnarray}
Now we sum over a complete orthonormal system (CONS) of $\mathcal{X}_0$, $\{f_{j}\}_{j}$ to obtain,
\begin{eqnarray*}
\|v_{t}^{\epsilon}\|_{\mathcal{X}_{0}}^{2} &=& 2a(\epsilon)\int_{0}^{t}\int_{U}<v_{s}^{\epsilon},G_{s}^{\epsilon}(a,\cdot,v_{s}^{\epsilon}(\cdot))>_{\mathcal{X}_{0}}
W(dsda)\\
&&\hspace{.3cm} + \int_{0}^{t} <v_{s}^{\epsilon},\Delta v_{s}^{\epsilon}>_{\mathcal{X}_{0}}ds \nonumber\\
&&\hspace{.3cm} + a(\epsilon)^2\int_{0}^{t}\int_{U}\int_{\mathbb{R}} G_{s}^{\epsilon}(a,y,v_{s}^{\epsilon}(y))^{2}J(y)dy\lambda(da)ds.
\end{eqnarray*}
By It$\hat{o}$'s formula,
\begin{eqnarray*}
&&\|v_{t}^{\epsilon}\|_{\mathcal{X}_{0}}^{2p} \\
&&= 2a(\epsilon)p\int_{U}\int_{0}^{t}\|v_{s}^{\epsilon}\|_{\mathcal{X}_{0}}^{2(p-1)}<v_{s}^{\epsilon},G_{s}^{\epsilon}
(a,\cdot,v_{s}^{\epsilon})
>_{\mathcal{X}_{0}}W(dsda)\\
&&\hspace{.3cm} +
\int_{0}^{t}p\|v_{s}^{\epsilon}\|_{\mathcal{X}_{0}}^{2(p-1)} <v_{s}^{\epsilon},\Delta v_{s}^{\epsilon}>_{\mathcal{X}_{0}}ds \\
&&\hspace{.3cm} + a(\epsilon)^2p\int_{0}^{t}\|v_{s}^{\epsilon}\|_{\mathcal{X}_{0}}^{2(p-1)}\int_{U}\int_{\mathbb{R}} G_{s}^{\epsilon}(a,y,v_{s}^{\epsilon}(y))^{2}J(y)dy\lambda(da)ds\\
&&\hspace{.3cm} + a(\epsilon)p(p-1)\int_{0}^{t}\int_{U} \|v_{s}^{\epsilon}\|_{\mathcal{X}_{0}}^{2(p-2)}<v_{s}^{\epsilon},G_{s}^{\epsilon}(a,\cdot,v_{s}^{\epsilon}(\cdot))
>_{\mathcal{X}_{0}}^{2}\lambda(da)ds.
\end{eqnarray*}
Similar to Kurtz and Xiong \cite{KX}, we can prove that
\[
-\int_{\mathbb{R}} v_{s}^{\epsilon}(y)\left(v_{s}^{\epsilon}\right)'(y)J'(y)dy = \frac{1}{2}\int_{\mathbb{R}} v_{s}^{\epsilon}(y)^{2}J''(y)dy\leq K\|v_{s}^{\epsilon}\|_{\mathcal{X}_{0}}^{2}\]
and
\begin{eqnarray*}
<v_{s}^{\epsilon},\Delta v_{s}^{\epsilon}>_{\mathcal{X}_{0}}
&\le& -\int_{\mathbb{R}}\left(v_{s}^{\epsilon}\right)'(y) v_{s}^{\epsilon}(y)J'(y)dy\\
&\leq& K\|v_{s}^{\epsilon}\|_{\mathcal{X}_{0}}^{2}.
\end{eqnarray*}
Hence with the help of Doob's and Burkholder-Davis-Gundy inequalities we have,
\begin{eqnarray*}
&&\mathbb{E}\sup_{0\leq s\leq t}\|v_{s}^{\epsilon}\|_{\mathcal{X}_{0}}^{2p} \\
&&\leq Ka(\epsilon)\mathbb{E}\left(\int_{0}^{t}\int_{U}\|v_{s}^{\epsilon}\|_{\mathcal{X}_{0}}^{4(p-1)}\left<v_{s}^{\epsilon},
G_{s}^{\epsilon}(a,y,v_{s}^{\epsilon}(y))\right>_{\mathcal{X}_{0}}^{2}ds\lambda(da)\right)^{\frac{1}{2}}\\
&&\hspace{.3cm}+ K\mathbb{E}\int_{0}^{t}\|v_{s}^{\epsilon}\|_{\mathcal{X}_{0}}^{2p}ds \\
&&\hspace{.3cm}+ Ka(\epsilon)\mathbb{E}\int_{0}^{t}\int_{U}\int_{\mathbb{R}} \|v_{s}^{\epsilon}\|_{\mathcal{X}_{0}}^{2(p-1)}G_{s}^{\epsilon}(a,y,v_{s}^{\epsilon}(y))^{2}J(y)dyds\lambda(da)\\
&&\hspace{.3cm}+ Ka(\epsilon)\mathbb{E}\int_{0}^{t}\int_{U}\|v_{s}^{\epsilon}\|_{\mathcal{X}_{0}}^{2(p-2)}\left<v_{s}^{\epsilon},
G_{s}^{\epsilon}(a,y,v_{s}^{\epsilon}(y))\right>_{\mathcal{X}_{0}}^{2}ds\lambda(da).
\end{eqnarray*}
Now we apply H$\ddot{o}$lder's inequality and (\ref{MDP con2}) to arrive at
\begin{eqnarray*}
\mathbb{E}\sup_{0\leq s\leq t}\|v_{s}^{\epsilon}\|_{\mathcal{X}_{0}}^{2p} &\leq&
Ka(\epsilon)\mathbb{E}\left(\int_{0}^{t}\|v_{s}^{\epsilon}\|_{\mathcal{X}_{0}}^{4p}
ds\right)^{1/2}
+ K\mathbb{E}\int_{0}^{t}\|v_{s}^{\epsilon}\|_{\mathcal{X}_{0}}^{2p}ds\\
&\le&Ka(\epsilon)\mathbb{E}\sup_{0\le s\le t}\|v_{s}^{\epsilon}\|_{\mathcal{X}_{0}}^p \left(\int_{0}^{t}\|v_{s}^{\epsilon}\|_{\mathcal{X}_{0}}^{2p}
ds\right)^{1/2} \\
&&\hspace{.2cm}+ K\mathbb{E}\int_{0}^{t}\|v_{s}^{\epsilon}\|_{\mathcal{X}_{0}}^{2p}ds\\
&\le&\frac12\mathbb{E}\sup_{0\leq s\leq t}\|v_{s}^{\epsilon}\|_{\mathcal{X}_{0}}^{2p}
+K_1\mathbb{E}\int_{0}^{t}\|v_{s}^{\epsilon}\|_{\mathcal{X}_{0}}^{2p}ds.
\end{eqnarray*}
The conclusion then follows from Gronwall's inequality.
\end{proof}

For our results, we also apply the lemma given below, the proof of which we have provided in \cite{me}.
\begin{lemma}\label{lemma2}
Let $\{X^{\epsilon}_t(y)\}$ be a family of random fields and
suppose $\beta_1\in (\beta_{0},\beta)$. If there exist constants
$n,\;q,\;K>0$ such that
\begin{equation}\label{kolmogorov}
\mathbb{E}\left|X^{\epsilon}_{t_1}(y_1)-X^{\epsilon}_{t_2}(y_2)\right|^n
\leq
Ke^{n\beta_1(|y_1|\vee|y_2|)}\left(|y_1-y_2|+|t_1-t_2|\right)^{2+q},
\end{equation}
then there exists a constant $\alpha >0$ such that
\begin{equation}\label{Holder}
\sup_{\epsilon>0} \mathbb{E}\left|\sup_{m} \sup_{t_{i} \in [0,1],
|y_{i}|\leq m,i=1,2}
\frac{\left|X^{\epsilon}_{t_1}(y_1)-X^{\epsilon}_{t_2}(y_2)\right|}{\left(|y_1-y_2|+|t_1-t_2|\right)^{\alpha}}e^{-\beta
m}\right|^n < \infty.
\end{equation}
As a consequence $X_{.}^{\epsilon} \in \mathcal{C}\left([0,1];\mathbb{B}_{\beta}\right)$ a.s. Furthermore, if condition (\ref{kolmogorov}) holds and
$\displaystyle\sup_{\epsilon>0}\mathbb{E}\left|X^{\epsilon}_{t_{0}}(y_{0})\right|^n
< \infty$ for some $(t_{0},y_{0}) \in [0,1]\times \mathbb{R}$, then
\begin{equation}\label{bdd}
\sup_{\epsilon>0}\mathbb{E}\left|\displaystyle \sup_{(t,y)\in [0,1]\times
\mathbb{R}} e^{-\beta|y|} |X_{t}^{\epsilon}(y)|\right|^n<\infty,
\end{equation}
and the family $\{X_{.}^{\epsilon}\}$ is tight in $\mathcal{C}\left([0,1];\mathbb{B}_{\beta}\right)$.
\end{lemma}

\begin{lemma}\label{lemm2}
The solution to SPDE (\ref{MDP SPDE}) takes values in $\mathcal{C}([0,1]; \mathbb{B}_{\beta})$.
\end{lemma}
\begin{proof}
First we need the following inequalities established in \cite{me}:

\begin{equation}\label{PP1}
P_{1}:= p_{t-s}(y_{1}-x)-p_{t-s}(y_{2}-x),
\end{equation}

\begin{equation}\label{PP2}
P_{2}:=p_{t_{1}-s}(y-x)-p_{t_{2}-s}(y-x),
\end{equation}

\begin{equation}\label{P1}
\int_{\mathbb{R}}|P_{1}|^{2}e^{2\beta_{1}|x|}dx\leq K e^{2\beta_{1}(|y_{1}|\vee |y_{2}|)} (t-s)^{-(\frac{1}{2}+\alpha)}|y_{1}-y_{2}|^{\alpha},
\end{equation}

\begin{equation}\label{P2}
\int_{0}^{t_{1}}\int_{\mathbb{R}} |P_{2}|^{2}e^{2\beta_{1}|x|}dxds\leq K e^{2\beta_{1}|y|} |t_{1}-t_{2}|^{\alpha},
\end{equation}
and
\begin{equation}\label{p2}
\int_{t_{1}}^{t_{2}} \int_{\mathbb{R}} p_{t_{2}-s}^{2}(y-x)e^{2\beta_{1}|x|}dxds\leq K|t_{1}-t_{2}|^{\frac{\alpha}{2}} e^{2\beta_{1}|y|}.
\end{equation}
We proceed by demonstrating two cases. In case one, we fix $t\in [0,1]$ and let $y_{1},y_{2}\in \mathbb{R}$ be arbitrary such that $|y_{i}|\leq m$ for all $i=1,2$. Applying Burkholder-Davis-Gundy and H$\ddot{o}$lder's inequalities, for $n>0$ we obtain,
\begin{eqnarray*}
&&\mathbb{E}\left|v_{t}^{\epsilon}(y_{1})-v_{t}^{\epsilon}(y_{2})\right|^{n}\\
&=& \mathbb{E}\left|a(\epsilon)\int_{0}^{t}\int_{U}\int_{\mathbb{R}} P_{1} G_{s}^{\epsilon}(a,x,v_{s}^{\epsilon}(x))W(dsda)dx\right|^{n}\\
&\leq& K\mathbb{E}\left(a(\epsilon)^{2}\int_{0}^{t}\int_{U}\left(\int_{\mathbb{R}}P_{1}G_{s}^{\epsilon}(a,x,v_{s}^{\epsilon}(x))dx\right)^{2}dsda
\right)^{n/2}\\
&\leq& K \mathbb{E}\left(a(\epsilon)^{2}\int_{0}^{t}\int_{U}\int_{\mathbb{R}} P_{1}^{2}e^{2\beta_{1}|x|}dx\int_{\mathbb{R}}G_{s}^{\epsilon}(a,x,v_{s}^{\epsilon}(x))^{2}e^{-2\beta_{1}|x|}dx\lambda(da)ds\right)^{n/2}\\
&\leq&  K \mathbb{E}\left(a(\epsilon)^{2}\int_{0}^{t}\int_{\mathbb{R}} P_{1}^{2}e^{2\beta_{1}|x|}dx\int_{\mathbb{R}}\left(1+ v_{s}^{\epsilon}(x)^{2}+ e^{2\beta_{0}|x|}\right)e^{-2\beta_{1}|x|}dxds\right)^{n/2}.
\end{eqnarray*}
By (\ref{P1}) we have,
\begin{eqnarray}\label{t fixed}
&&\mathbb{E}|v_{t}^{\epsilon}(y_{1})-v_{t}^{\epsilon}(y_{2})|^{n} \\
&\leq& K \mathbb{E}\left(\int_{0}^{t} e^{2\beta_{1}(|y_{1}|\vee |y_{2}|)}(t-s)^{-(\frac{1}{2}+\alpha)}  |y_{1}-y_{2}|^{\alpha} \int_{\mathbb{R}}v_{s}^{\epsilon}(x)^{2}e^{-2\beta_{1}|x|}dxds\right)^{n/2}\nonumber \\
&\leq& \bar{M}Ke^{n\beta_{1}(|y_{1}|\vee |y_{2}|)}|y_{1}-y_{2}|^{\frac{n\alpha}{2}}. \nonumber
\end{eqnarray}
For the second case, we consider $y\in \mathbb{R}$ to be fixed and assume $t_{1},t_{2}\in [0,1]$ to be arbitrary, then by (\ref{P2}) and (\ref{p2}),
\begin{eqnarray}\label{y fixed}
&&\mathbb{E}\left|v_{t_{1}}^{\epsilon}(y)-v_{t_{2}}^{\epsilon}(y)\right|^{n} \\
&\leq& K \mathbb{E}\left|a(\epsilon)\int_{0}^{t_{1}}\int_{\mathbb{R}}\int_{U} P_{2} G_{s}^{\epsilon}(a,x,v_{s}(x)) W(dsda)dx\right|^{n}\nonumber\\
&&\hspace{.2cm}+ K \mathbb{E} \left|a(\epsilon)\int_{t_{1}}^{t_{2}} \int_{\mathbb{R}} \int_{U} p_{t_{2}-s}(y-x)G_{s}^{\epsilon}(a,x,v_{s}(x)) W(dsda)dx\right|^{n}\nonumber\\
&\leq& K\mathbb{E}\left|\int_{0}^{t_{1}} \int_{\mathbb{R}} P_{2}^{2} e^{2\beta_{1}|x|}dx\int_{\mathbb{R}} \left(K+ v_{s}^{\epsilon}(x)^{2}\right)e^{-2\beta_{1}|x|}dxds\right|^{\frac{n}{2}}\nonumber\\
&&\hspace{.2cm} + K\mathbb{E}\left|\int_{t_{1}}^{t_{2}} \int_{\mathbb{R}} p_{t_{2}-s}^{2}(y-x)e^{2\beta_{1}|x|}dx\int_{\mathbb{R}} \left(K + v_{s}^{\epsilon}(x)^{2}\right)e^{-2\beta_{1}|x|}dxds\right|^{\frac{n}{2}}\nonumber\\
&\leq& \bar{M}K\left|\int_{0}^{t_{1}}\int_{\mathbb{R}} P_{2}^{2}e^{2\beta_{1}|x|}dx\right|^{\frac{n}{2}}
+ \bar{M}K\left|\int_{t_{1}}^{t_{2}}\int_{\mathbb{R}} p_{t_{2}-s}^{2}(y-x)e^{2\beta_{1}|x|}dxds\right|^{\frac{n}{2}}\nonumber \\
&\leq& K e^{n\beta_{1}|y|}|t_{1}-t_{2}|^{\frac{\alpha n}{2}} + Ke^{n\beta_{1}|y|}|t_{1}-t_{2}|^{\frac{\alpha n}{4}}\nonumber \\
&\leq& Ke^{n\beta_{1}|y|}|t_{1}-t_{2}|^{\frac{\alpha n}{4}}, \nonumber
\end{eqnarray}
where in the last step we have used the fact that $\left|t_{1}-t_{2}\right| <1$.
\end{proof}

We now prove Theorem 1 by applying a technique offered by Budhiraja, et al in \cite{Bud}. To match their setup, we write SPDE (\ref{MDP SPDE}) as an infinite sum of independent Brownian motions as follows. Suppose $\{\phi_{j}\}_{j}$ is a CONS of $L^{2}(U,\mathcal{U},\lambda)$ then,
\begin{equation}
B_{t}^{j}:= \int_{0}^{t}\int_{U}\phi_{j}(a)W(dsda), \hspace{1cm} j=1,2,...
\end{equation}
is a sequence of independent Brownian motions by L$\acute{e}$vy's characterization of Brownian motions. We can then present SPDE (\ref{MDP SPDE}) in the following form,
\begin{equation}\label{sumMDP}
v_{t}^{\epsilon}(y)= a(\epsilon) \sum_{j}\int_{0}^{t} G_{s}^{\epsilon,j}(y,v_{s}^{\epsilon}(y)) dB_{s}^{j} + \frac{1}{2}\int_{0}^{t}\Delta v_{s}^{\epsilon}(y)ds,
\end{equation}
where
\begin{equation}
G_{s}^{\epsilon,j}(y,v):= \int_{U}G_{s}^{\epsilon}(a,y,v)\phi_{j}(a)\lambda(da).
\end{equation}
Similarly, the controlled PDE (\ref{controlled}) can be written as
\begin{equation}
v_{t}(y)= \sum_{j}\int_{0}^{t}\int_{U}G(a,y,u_{s}^{0}(y)) k_{s}^{j}\phi_{j}(a)\lambda(da)ds + \frac{1}{2}\int_{0}^{t}\Delta v_{s}(y)ds,
\end{equation}
where
\begin{equation*}
k_{s}^{j}:= \int_{U} h_{s}(a)\phi_{j}(a)\lambda(da).
\end{equation*}
By the same argument as in \cite{Xio}, SPDE (\ref{sumMDP}) has a strong solution so there exists a map $g^{\epsilon}: \mathbb{B}_{\alpha, \beta_{0}} \times S \rightarrow \mathcal{C}\left([0,1];\mathbb{B}_{\beta}\right)$ such that $v^{\epsilon}=g^{\epsilon}\left(a(\epsilon)B\right)$ where $B=\{B_{t}^{j}\}$.
 We now define
\begin{equation}\label{N}
\mathcal{S}^{N}(\ell_{2}):= \{k \in L^{2}([0,1], \ell_{2}):\int_{0}^{1}\|k_{s}\|_{\ell_{2}}^{2}ds\leq N\}.
\end{equation}
To verify the assumption imposed by Theorem 6 in \cite{Bud} let $\{k^{\epsilon}\}$ be a family of random variables taking values in $\mathcal{S}^{N}(\ell_{2})$ such that $k^{\epsilon}\rightarrow k$ in distribution as $\epsilon\rightarrow 0$ and consider the SPDE,
\begin{eqnarray}\label{solutionSPDE}
v_{t}^{\theta,\epsilon}(y)&=& \theta \sum_{j} \int_{0}^{t}\int_{\mathbb{R}} p_{t-s}(y-x)G_{s}^{\epsilon,j}(x,v_{s}^{\theta, \epsilon}(x)) dB_{s}^{j}dx  \\
&& \hspace{.3cm} + \sum_{j}\int_{0}^{t}\int_{\mathbb{R}} p_{t-s}(y-x)G_{s}^{\epsilon,j}(x,v_{s}^{\theta, \epsilon}(x))k_{s}^{\epsilon,j}dxds. \nonumber
\end{eqnarray}
We establish the tightness of $\{v^{\theta,\epsilon}\}$ as follows.

\begin{lemma}
$v_{t}^{\theta,\epsilon}(y)$ is tight in $\mathcal{C}([0,1],\mathbb{B}_{\beta})$.
\end{lemma}

\begin{proof}
  According to Lemma \ref{lemma2}, to achieve the tightness for $\{v_{t}^{\theta, \epsilon}\}$, it is sufficient to show
  estimate (\ref{kolmogorov}) for $v_{t}^{\theta,\epsilon}$ and verify that $\sup_{\epsilon>0} \mathbb{E}\left|v_{t_{0}}^{\theta,\epsilon}(y_{0})\right|^{n}<\infty$ for some $(t_{0},y_{0})\in [0,1]\times \mathbb{R}$. Following the same steps as in the proof of lemma \ref{lem1}, we have
  \begin{equation}\label{MM}
  \tilde{M}:= \sup_{0<\epsilon<1}\mathbb{E}\sup_{0\leq s\leq 1}\left(\int_{\mathbb{R}} \left|v_{s}^{\theta,\epsilon}(x)\right|^{2}e^{-2\beta_{1}|x|}dx\right)^{n}<\infty.
  \end{equation}
  Notice that estimate (\ref{kolmogorov}) can be attained for the first term on the right hand side of (\ref{solutionSPDE}) by exactly the same calculations done in Lemma \ref{lemm2} with the use of $\tilde{M}$ given in (\ref{MM}) instead of $\bar{M}$ of Lemma \ref{lem1}. Thus, we focus on finding estimate (\ref{kolmogorov}) for
  \begin{equation*}
  \tilde{v}_{t}^{\theta,\epsilon}(y):= \sum_{j} \int_{0}^{t}\int_{\mathbb{R}} p_{t-s}(y-x)G_{j}^{\epsilon,j}(x, v_{s}^{\theta,\epsilon}(x))k_{s}^{\epsilon,j}dxds.
  \end{equation*}

 Using the same method used in the proof of Lemma \ref{lemm2}, we begin by fixing $t\in [0,1]$ and assuming $y_{1},y_{2}$ to be any real numbers such that $|y_{i}|\leq m$ for $i=1,2$ and $m\in \mathbb{N}$. Recall
\begin{equation*}
P_{1}:= p_{t-s}(y_{1}-x)-p_{t-s}(y_{2}-x).
\end{equation*}
With the help of Cauchy-Schwartz inequality, (\ref{MDP con2}), and our result (\ref{P1}), we obtain the estimates below,
\begin{eqnarray*}
&&\mathbb{E}\left|\tilde{v}_{t}^{\theta,\epsilon}(y_{1})-\tilde{v}_{t}^{\theta,\epsilon}(y_{2})\right|^{n}\\
&=& \mathbb{E}\left|\int_{0}^{t}\int_{\mathbb{R}} P_{1}\sum_{j}G_{s}^{\epsilon,j}(x,v_{s}^{\theta,\epsilon}(x))k_{s}^{\epsilon,j}dxds\right|^{n}\\
&\leq& \mathbb{E}\left|\int_{0}^{t} \int_{\mathbb{R}} P_{1}\left(\sum_{j}G_{s}^{\epsilon,j}(x,v_{s}^{\theta,\epsilon}(x))^{2}\right)^{\frac{1}{2}}\|k_{s}^{\epsilon}\|_{\ell_{2}}dxds\right|^{n}\\
&\leq& \mathbb{E}\left|\left(\int_{0}^{t} \left(\int_{\mathbb{R}} P_{1}\sqrt{K\left(1+v_{s}^{\theta,\epsilon}(x)^{2}+e^{2\beta_{0}|x|}\right)}dx\right)^{2}ds\right)^{\frac{1}{2}}\left(\int_{0}^{t}\|k_{s}
^{\epsilon}\|_{\ell_{2}}^{2}ds\right)^{\frac{1}{2}}\right|^{n}\\
&\leq& \mathbb{E} \left|\int_{0}^{t}\left(\int_{\mathbb{R}}P_{1}\sqrt{K\left(1+v_{s}^{\theta,\epsilon}(x)^{2}+e^{2\beta_{0}|x|}\right)}dx\right)^{2}dx\right|^
{\frac{n}{2}}N^{\frac{n}{2}} \\
&\leq& K e^{n\beta_{1}(|y_{1}|\vee |y_{2}|)} |y_{1}-y_{2}|^{\frac{\alpha n}{2}},
\end{eqnarray*}
where $N>0$ is the constant given by (\ref{N}). Furthermore, the case for $0\leq t_{1}<t_{2}\leq 1$ arbitrary and $y\in \mathbb{R}$ fixed can be given by
\begin{eqnarray*}
\mathbb{E}\left|\tilde{v}_{t_{1}}^{\theta,\epsilon}(y)-\tilde{v}_{t_{2}}^{\theta,\epsilon}(y)\right|^{n}&\leq& K\mathbb{E} \left|\int_{t_{1}}^{t_{2}} \int_{\mathbb{R}} P_{2} \sum_{j} G_{s}^{\epsilon,j}(x,v_{s}(x))k_{s}^{\epsilon,j}dxds\right|^{n}\\
&&\hspace{.1cm}  + K\mathbb{E} \left|\int_{0}^{t_{1}} \int_{\mathbb{R}} p_{t_{2}-s}^{2}(y-x) \sum_{j} G_{s}^{\epsilon,j}(x,v_{s}(x))k_{s}^{\epsilon,j}dxds\right|^{n}\\
&\leq& K e^{n\beta_{1}|y|} |t_{1}-t_{2}|^{\frac{n\alpha}{4}},
\end{eqnarray*}
where,
\begin{equation*}
P_{2}:= p_{t_{1}-s}(y-x)-p_{t_{2}-s}(y-x).
\end{equation*}
\end{proof}
Thus, $\{v_{t}^{\theta,\epsilon}\}$ is tight and for the assumption of Theorem 6 of \cite{Bud} to be satisfied we let $\theta = 0$ for its first part and $\theta= a(\epsilon)$ for the second part and apply the Prohorov Theorem and so by their Theorem 6, our Theorem \ref{them1} can be deduced.

\section{Moderate Deviations for SBM and FVP}
We devote this section to the proofs of Theorems \ref{them2} and \ref{them3}. Recall
\begin{equation}\label{w}
\omega_{t}^{\epsilon}(dy):= \frac{a(\epsilon)}{\sqrt{\epsilon}} \left(\mu_{t}^{\epsilon}(dy)-\mu_{t}^0(dy)\right),
\end{equation}
 where in the case of SBM, $u_{t}^{\epsilon}(y):= \int_{0}^{y}\mu_{t}^{\epsilon}(dx)$. Then based on (\ref{centered}), we can write $v_{t}^{\epsilon}(y):= \int_{0}^{y}\omega_{t}^{\epsilon}(dx)$. Similarly, for FVP we have $u_{t}^{\epsilon}(y):=\int_{-\infty}^{y}\mu_{t}^{\epsilon}(dx)$ which gives $v_{t}^{\epsilon}(y):=\int_{-\infty}^{y}\omega_{t}^{\epsilon}(dx)$. Analogous to Lemma 6 of \cite{me} we have that for the set of functions with finite variations, $\mathcal{A}$, the map $\xi: \mathbb{B}_{\beta}\cap \mathcal{A}\rightarrow \mathcal{M}^S_{\beta}(\mathbb{R})$ given as $\xi(u)(B)= \int 1_{B}(y)du(y)$ for all $B\in \mathbb{B}(\mathbb{R})$ is continuous. Therefore, map
  $\eta: \mathcal{C}\left([0,1]; \mathbb{B}_{\beta}\right)\rightarrow \mathcal{C}\left([0,1]; \mathcal{M}^S_{\beta}(\mathbb{R})\right)$ defined as $\eta(v)_{t}=\xi(v_{t})$ is also continuous. Since SBM and FVP can be written as $\omega_{t}^{\epsilon}(dy)= \eta(v^\epsilon)_{t}([0,y])$ and $\omega_{t}^{\epsilon}(dy)= \eta(v^\epsilon)_{t}((-\infty, y])$ respectively, then in both cases $\omega_{t}^{\epsilon}(y)$ is a continuous function of $v_{t}^{\epsilon}$. Based on our LDP result for $v_{t}^{\epsilon}$ given in Theorem \ref{them1}, we can conclude by the contraction principle that $\{\omega_{t}^{\epsilon}\}$ also satisfies the LDP for both models.

 Our remaining task is to identify an explicit representation of the models' MDP rate functions. According to the contraction principle, rate functions for SBM and FVP are given by $\inf\left\{I(u):u\in \eta^{-1}(\omega)\right\}$. Since $\eta$ is injective, we then aim to find the rate functions following the form given by (\ref{rate}).

As for SBM, formulation (\ref{SBM}) satisfies the general SPDE (\ref{SPDE}) with the following properties,
\begin{equation*}
   U= \mathbb{R}, \hspace{.3cm}\lambda(da)=da, \hspace{.3cm}G(a,y,u)= 1_{0\leq a\leq u}+ 1_{u\leq a \leq 0},
  \end{equation*}
then using the controlled PDE (\ref{controlled}) we have,
\begin{eqnarray*}
&&<\omega_{t},f>= <\partial_{x}v_{t},f> = -<v_{t},f'>\\
&=& -\int_{0}^{t}\int_{0}^{\infty}\int_{0}^{u_{s}^{0}(y)}h_{s}(a)f'(y)dadyds - \int_{0}^{t}\int_{-\infty}^{0}\int_{u_{s}^{0}(y)}^{0}h_{s}(a)f'(y)dadyds \\
&&\hspace{.2cm} - \int_{0}^{t}<\frac{1}{2}\Delta v_{s},f'>ds\\
&=& \int_{0}^{t}\int_{0}^{\infty}h_{s}(a)f((u_{s}^{0})^{-1}(a))dads- \int_{0}^{t}\int_{-\infty}^{0}h_{s}(a)f((u_{s}^{0})^{-1}(a))dads\\
&&\hspace{.2cm}+ \int_{0}^{t}<\frac{1}{2}\omega_{s},\Delta f>ds\\
&=& \int_{0}^{t}\int_{0}^{\infty}h_{s}(u_{s}^{0}(y))f(y)du_{s}^{0}(y)ds-\int_{0}^{t}\int_{-\infty}^{0}h_{s}(u_{s}^{0}(y))f(y)
du_{s}^{0}(y)ds \\
&&\hspace{.2cm} +\int_{0}^{t} <\frac{1}{2}\Delta^{*} \omega_{s},f>ds\\
&=& \int_{0}^{t}<h_{s}(u_{s}^{0})sgn(.)\mu_{s}^{0},f>ds+\frac{1}{2}\int_{0}^{t}<\Delta^{*}\omega_{s},f>ds.
\end{eqnarray*}
Thus,
\begin{equation*}
h_{t}(u_{t}^{0}(y))sgn(y)= \frac{d\left(\dot{\omega}_{t}-\frac{1}{2}\Delta^{*}\omega_{t}\right)}{d\mu_{t}^{0}}(y).
\end{equation*}
Notice that
\begin{equation*}
\int_{\mathbb{R}} |h_{t}(a)|^{2}da = \int_{\mathbb{R}} |h_{t}(u_{t}^{0}(y))|^{2}du_{t}^{0}(y)= \int_{\mathbb{R}} |h_{t}(u_{t}^{0}(y))|^{2}d\mu^0_{t}(y).
\end{equation*}
Letting the right hand side of (\ref{rate4sbm}) be denoted as $I_{0}(\mu)$, if $I(\mu)<\infty$ then $I(\mu)$ given in (\ref{rate}) with $U=\mathbb{R}$ is equal to $I_{0}(\mu)$. For the case $I_{0}(\mu)<\infty$ we can reverse the above calculations to obtain $I_{0}(\mu)= I(\mu)$.

Similarly for FVP, since FVP satisfies the general SPDE (\ref{SPDE}) with
\begin{equation*}
   U= [0,1],\hspace{.3cm} \lambda(da)=da, \hspace{.3cm} G(a,y,u)= 1_{a<u}-u,
  \end{equation*}
then
\begin{eqnarray*}
&&<\omega_{t},f>= -<v_{t},f'>\\
&=& -\int_{0}^{t}\int_{\mathbb{R}}\int_{0}^{u_{s}^{0}(y)} h_{s}(a)f'(y)dadyds \\
&&\hspace{.2cm} + \int_{0}^{t}\int_{\mathbb{R}}\int_{0}^{1}u_{s}^{0}(y)h_{s}(a)f'(y)dadyds
 - \int_{0}^{t}<\frac{1}{2}\Delta v_{s}(y),f'>ds\\
&=& \int_{0}^{t}<h_{s}(u_{s}^{0})\mu_{s}^{0},f>ds -\int_{0}^{t}<\int_{0}^{1}h_{s}(a)da\mu_{s}^{0},f>ds \\
&&\hspace{.2cm} +\int_{0}^{t}<\frac{1}{2}\Delta^{*}\omega_{s},f>ds.
\end{eqnarray*}
Thus,
\begin{equation*}
\dot{\omega}_{t}-\frac{1}{2}\Delta^*\omega_{t}= h_{t}(u_{t}^{0}(y))\mu_{t}^{0}- \int_{0}^{1}h_{t}(a)da \mu_{t}^{0}.
\end{equation*}
Our goal is to find the infimum of $\int_{0}^{1}\left|h_{s}(a)\right|^{2}da$ over $h_{s}(a)$ satisfying (\ref{controlled}). We note that if $h$ satisfies (\ref{controlled}) then $g_{s}(a):=h_{s}(a)-\int_{0}^{1}h_{s}(a)da$ also satisfies the same equation. It is well-known that the second moment is minimized when it is centralized. Therefore, we replace $h_{s}(a)$ by $g_{s}(a)$ in the definition of the rate function, and obtain it as,
\begin{equation*}
\int_{0}^{1}|g_{s}(a)|^{2}da= \int_{0}^{1}\left|\frac{d\left(\dot{\omega}_{t}-\frac{1}{2}\Delta^{*}\omega_{t}\right)}{d\mu_{t}^{0}}(y)
\right|^{2}d\mu_{t}^{0}(y),
\end{equation*}
in (\ref{rate}) to arrive at (\ref{rate4fvp}) for the case $I(v)<\infty$ and based on a similar argument as in the case of SBM we obtain (\ref{rate4fvp}). Thus, MDP is proved for the two models.

\end{document}